\documentclass[12pt, reqno]{amsart}
\allowdisplaybreaks[3]
\usepackage[left=3cm,right=3cm,top=3cm,bottom=3cm]{geometry}

\usepackage{tikz-cd}
\usepackage{amsthm}
\usepackage{amscd}
\usepackage{amsmath}
\usepackage{amsfonts}
\usepackage{amssymb}
\usepackage{graphicx}
\usepackage{amsopn}
\usetikzlibrary{arrows,shapes,snakes,automata,backgrounds,petri}
\usepackage{textcomp}
\usepackage[all]{xy}

\usepackage[utf8]{inputenc}
\usepackage[T2A]{fontenc}
\usepackage{amsmath,amssymb,amsthm}
\usepackage[english]{babel}

\usepackage{algorithm}
\usepackage{algorithmic}
\usepackage{mathtools}

\usepackage{hyperref}

\theoremstyle{theorem}
\newtheorem{theorem}{Theorem}
\newtheorem{corollary}[theorem]{Corollary}

\newtheorem{lemma}[theorem]{Lemma}

\theoremstyle{definition}
\newtheorem{definition}[theorem]{Definition}
\newtheorem{construction}[theorem]{Construction}

\newtheorem{example}[theorem]{Example}
\newtheorem{fact}[theorem]{Fact}
\newtheorem{remark}[theorem]{Remark}

\usepackage{verbatim}

\def\C{\mathcal{C}}
\def\D{\mathcal{D}}
\def\A{\mathcal{A}}
\def\I{\mathcal{I}}
\def\F{\mathcal{F}}
\def\V{\mathcal{V}}
\def\G{\mathcal{G}}

\def\T{\mathcal{T}}

\def\B{\mathcal{B}}

\def\Im{{\rm Im}}

\def\cat{{\rm cat}}

\def\id{{\rm id}}
\def\Z{\mathbb{Z}}

\def\D{\mathcal{D}}
\def\T{\mathcal{T}}
\def\A{\mathcal{A}}

\def\H{{\rm H}}

\def\pt{{\rm pt}}

\DeclareMathOperator{\ord}{ord}
\DeclareMathOperator{\im}{Im}
\DeclareMathOperator{\Coker}{Coker}
\DeclareMathOperator{\Cone}{Cone}
\DeclareMathOperator{\Tor}{Tor}

\newcommand{\low}[2]{{_{\lceil}#1_{\rceil #2}}}
\newcommand{\Vect}{\mathbf{Vect}}
\newcommand{\Abel}{\mathbf{Abel}}

\newcommand{\Cochain}{\mathbf{Cochain}}
\newcommand{\ko}{\Bbbk}

\numberwithin{theorem}{section}
\allowdisplaybreaks[3]

\newcommand{\inc}[2]{{[#1\!:\!#2]}}
\begin{document}

\title{Simplifications of finite spaces equipped with sheaves}

\author{Artem Malko}
\address{National Research University Higher School of Economics, Moscow 101000, Russia}
\email{metraoklam@gmail.com}

\maketitle

\begin{abstract}
    Following the classical results of Stong, we introduce a cohomological analogue of a core of a finite sheaved topological space and propose an algorithm for simplification in this category. In particular we  generalize the notion of beat vertices and show that if a vertex of a sheaved space has topologically acyclic downset (with trivial coefficients), then its removal preserves the sheaf cohomology.
\end{abstract}

\section{Introduction}

\subsection{Overview}

This paper is a part of a big project aimed at development of algorithmic approaches in sheaf theory. Generally, sheaf theory is a huge part of homological algebra finding lots of application in theoretical mathematics, especially in algebraic geometry and algebraic topology \cite{algem}. In these areas specialists usually consider and study infinite-dimensional sheaves on infinite topological spaces with function- or differential forms sheaves being one of examples. Of course, these sheaves are non-tractable from algorithmic point of view. However, in the last decade there has been an increasing interest in high-level homological methods in the areas of data analysis and more generally, in computer science. Simplicial homology and more general persistent homology are an important part of topological data analysis. \cite{comtop}, \cite{eltop}

At the same time, currently there is no available library to compute cohomology of finite-dimensional sheaves on finite topological spaces, despite the fact that these objects can be effectively encoded and the problem itself is algorithmically solvable. Yet, even at a theoretical level not so many optimal algorithms exists to deal with mathematically loaded data structures such as finite topological spaces, sheaves, differential complexes, etc. Some of them are implemented in SageMath~\cite{sagemath} and Macaulay2~\cite{macaulay2}, but there is no functionality of general sheaves on finite structures. The aim of our project is to develop approaches to effective calculations of sheaf cohomology.

In homotopy theory of finite topological spaces there is a notion of a core of a topological space, introduced by Stong \cite{stong}. Recently this notion found application in homology and persistent homology computations \cite{compute}: instead of computing homology of the given space, it is practically efficient to first reduce this space to its core, and then compute homology. In the current paper we propose a similar approach to simplify a given space with a given sheaf on it. We introduce the notion of a core of sheaved space and prove that cohomology of core is isomorphic to the cohomology of the original sheaf, see Theorem \ref{result_main} for precise formulation. In the following we plan to incorporate this approach in a C++ library which allows to compute cohomology.

\subsection{Motivation from machine learning}
Sheaves on cellular complexes are widely used in modern computer science and machine learning. One can recall Sheaf Neural Networks (SNNs) \cite{sheafnn} as a generalization of graph neural networks. The main problems that SNNs aim to address are over-smoothing and heterophily. Another approach is through sheaf diffusion models \cite{sheafdiff}, whose core idea involves the incorporation of sheaf Laplacians to enable anisotropic information flow across graphs. These models seek to capture both local and global interactions by allowing more refined control over how data propagates.

However, most of the current applications of Sheaf Neural Networks and sheaf diffusion models are focused on practical use cases, leaving substantial room for purely mathematical optimizations. For instance, one such area of improvement is enhancing the algorithmic efficiency of verifying the commutativity of diagrams. A sheaf on a poset can be described as a commutative diagram; more formally, these are equivalent categories.  We are currently working on another paper that proposes an asymptotically optimal algorithm for verifying diagram commutativity.

Another classic example of such an optimization applied in practice lies in the calculation of homology and Betti numbers of partially ordered sets (simplicial complexes). Non-constant sheaves over finite posets appear in the study of torus actions on manifolds \cite{gkmsheaves} \cite{baird}, where exact automatic cohomology computation supports purely theoretical results \cite{ayzen}.

\subsection{Related definitions}
Let $S$ be a finite partially ordered set. Recall that Hasse diagram representation of poset $S$ is oriented graph $G = (S, E)$ such that there exists edge from $u$ to $v$ iff $u < v$ and there is no $t$ such that $u < t < v$.

\begin{definition}
    An element $s \in S$ is called \textit{upbeat} if there is $a > s$ such that $b > s$ implies that $b \geq a$. An element $s \in S$ is called \textit{downbeat} if there is $a < s$ such that $b < s$ implies that $b \leq a$. An element $s \in S$ is a \textit{beat} point if it is either an upbeat point or a downbeat point.
\end{definition}

\begin{definition}
    Poset $S$ is called \textit{core} if it has no beat vertices. Poset $T$ is called \textit{core} of $S$ if $T$ is core and can be obtained from $S$ by consecutive deletion of beats.
\end{definition}

\begin{theorem}[R. Stong \cite{stong}] 
    All cores of $S$ are isomorphic.
\end{theorem}

Next theorem needs some topological notation. There are two natural ways of converting partially ordered set into topological space: an Alexandrov topology and geometric implementation as simplicial complex. They are linked by McCord Theorem \cite{mccord} that we will formulate later.

\begin{definition}
    An \textit{Alexandrov topology} \( X_S \) on a poset \( S \) is a topology in which the open sets are upper order ideals in \( S \). In other words, a set \( U \) is open in the topology \( X_S \) if and only if for any \( s \in U \) and \( t > s \), it holds that \( t \in U \).
\end{definition}

Let $K$ be a finite simplicial complex constructed w.r.t. to partial order given by poset $S$ and $|K|$ be its geometrical interpretation.

\begin{theorem}[M. McCord \cite{mccord}]
The topological spaces $|K|$ and $X_S$ are weakly equivalent. In particular, their singular (co)homology are isomorphic. 
\end{theorem}

\begin{theorem}[R. Stong \cite{stong}] 
    Let $S$ be a finite poset, denote by $T$ its core. Then $X_S$ is homotopy equivalent to $X_T$. 
\end{theorem}

\begin{remark}[R. Stong \cite{stong}] 
    Moreover, for two finite posets $S, T$ we have $X_S \simeq X_T$ if and only if the cores of $S$ and $T$ are isomorphic.
\end{remark}

This classic result, initially motivated by problem of classification of partially ordered sets, provides an algorithm for speeding up calculation of homology of partially ordered sets and simplicial complexes. This was clearly demonstrated by J. Boissonnat, S. Pritam and D. Pareek in 2018 \cite{compute}.

A homotopical analogue of our approach to the category of the finite ringed spaces can be found in \cite{ringedspaces}.

\subsection{Results}
We propose a generalization of the concept of cores of posets to posets equipped with a sheaf. In particular we introduce the notion of $\textit{cores}$ of posets with a sheaf and prove that all cores of $(S, \F)$ are isomorphic and have the same cohomology. Finally, we generalize the notion of beat vertex to acyclic downsets.

Let us introduce some algebraic notation. Let $\V$ be an abelian category, for example $\Abel$ (the category of abelian groups with group homomorphisms as morphisms) or $\ko \Vect$ (the category of vector spaces over a field $\ko$).

\begin{definition}
    Category $\mathrm{SheavedSpaces}$ consists of pairs $(S, \F)$ where $S$ is finite poset and $\F$ is a sheaf defined on $X_S$ with values in $\V$. Morphisms $\A: (S, \F) \to (T, \G)$ are then defined as $\A = (a, A)$, where $a: S \hookrightarrow T$, $A: a^{\ast} \G \to \F$.
\end{definition}

Now with category formalized we are ready to formulate main results.

\begin{definition}
    Let $(S, \F)$ be a sheaved space. An element $s \in S$ is called \textit{downbeat} if $s$ is downbeat in $S$. An element $s \in S$ is called \textit{upbeat} if $s$ is upbeat in $S$ and corresponding restriction map (i.e., the map that corresponds to the only edge outgoing from $s$ in Hasse diagram notation) is isomorphism. An element $s \in S$ is a \textit{beat} point if it is either an upbeat point or a downbeat point.
\end{definition}

\begin{theorem}
    Let $v$ be a beat element of a sheaved space $(S, \F)$. Denote by $S' = S \setminus \{ v \}$. The canonical inclusion $S' \to S$ induces the natural isomorphism

    $$H^{\ast}(S, \F) \cong H^{\ast}(S', \F|_{S'}) $$
\end{theorem}

\begin{definition}
    Sheaved space $(S, \F)$ is called \textit{core} if it has no beat elements. Sheaved space $(T, \G)$ is called \textit{core} of $(S, \F)$ if $(T, \G)$ is core and can be obtained from $(S, \F)$ by consecutive deletion of beats.
\end{definition}

\begin{theorem}
    All cores of $(S, \F)$ are isomorphic.
\end{theorem}

\begin{theorem} \label{result_main}
    Let $(T, \G)$ be a core of $(S, \F)$. The canonical inclusion $T \to S$ induces the natural isomorphism

    $$H^{\ast}(X_T, \G) \cong H^{\ast}(X_S, \F)$$
\end{theorem}

We also propose the generalization of the notion of beat vertices in this setting. In particular, we show that if a vertex of a sheaved space has an acyclic downset, then removing it preserves the cohomology.

\begin{theorem} \label{result_main_2}
    Let $(S, \F)$ be a sheaved space. Let be $s$ be such vertex that $H_{\star}(|S_{<s}|, \Z)$ are trivial (i.e., $|S_{<s}|$ has the homology of a point). Then

    $$H^{\ast}(X_S, \F) \cong H^{\ast}(X_{S'}, \F|_{S'})$$

    where $S' = S \setminus \{ s \}$.
\end{theorem}

\begin{remark}
    Theorem \ref{result_main_2} also applies to constant sheaf on a poset which gives isomorphism of ordinary cohomology of $S$ and $S'$ (see Theorem \ref{homologyStong}).
\end{remark}

Implications of this result are straightforward as it gives instrument for speeding up calculations of cohomology of sheaves. For instance, in practice, only the 0th Betti number is often used due to the limitations in algorithmic efficiency. Our work has the potential to shift this paradigm.

Applied researchers typically do not work with general sheaves, instead restricting themselves to cellular sheaves, leaving the algorithmic aspects largely underexplored. This gap presents another key contribution of this paper: we aim to introduce a more general and formal setting into the applied domain, expanding the range of techniques and frameworks available for practical use. 

\subsection{Structure}

This paper is organized as follows. Section \ref{S2} contains all necessary preliminaries, definitions and constructions used. In Section \ref{S3} we show that deletion of one beat does not change cohomology of a sheaved space and, in Section \ref{S4}, we show that all cores are isomorphic and prove Theorem \ref{result_main}. In Section \ref{S5} we propose a generalization of beat collapses and discuss applications for constant sheaves.

\subsection{Acknowledgements} 
The author expresses his gratitude to his scientific advisor, A.~Ayzenberg, for his constant guidance and support throughout the development of this paper, and to I.~Spiridonov for useful discussions.

\section{Preliminaries} \label{S2}

\subsection{Partially ordered sets and sheaves.}

Let $S$ be a finite partially ordered set.

\begin{definition}
    Let $T,S$ be partially ordered sets. Mapping $f: S \to T$ is called \textit{morphism} if for any elements $s_1, s_2 \in S$ such that $s_1 \leq s_2$ holds that $f(s_1) \leq f(s_2)$.
\end{definition}

Let $\V$ be an abelian category, for example $\Abel$ or $\ko \Vect$.

\begin{definition}
    \textit{Presheaf} on topological space $X = (M, \Omega_X)$ with values in $\V$ is a contravariant functor $\F: \cat(\Omega_X \setminus \{ \varnothing \})^{op} \to \V$
\end{definition}
    
\begin{definition}
    Presheaf $\F$ is called \textit{sheaf} if it satisfies following axioms:

    $\bullet$ \textit{Locality}: Suppose $U$ is an open set, $\{ U_i \}_{i \in I}$ is an open cover of $U$ with $U_i \subseteq U$ for all $i \in I$, and $s, t \in \F(U)$ are sections. If $s|_{U_i} = t|_{U_i}$ for all $i \in I$, then $s = t$.

    $\bullet$ \textit{Gluing}: Suppose \( U \) is an open set, \( \{U_{i}\}_{i\in I} \) is an open cover of \( U \) with \( U_{i} \subseteq U \) for all \( i \in I \), and \( \{s_{i} \in \F(U_{i})\}_{i \in I} \) is a family of sections. If all pairs of sections agree on the overlap of their domains, that is, if \( s_{i}|_{U_{i}\cap U_{j}} = s_{j}|_{U_{i}\cap U_{j}} \) for all \( i, j \in I \), then there exists a section \( s \in \F(U) \) such that \( s|_{U_{i}} = s_{i} \) for all \( i \in I \).

\end{definition}

\begin{remark}
    Presheaves on topological space $X$ form a category $\mathrm{PreShvs}(X, \V)$ where morphisms are defined as natural transformations of functors. Sheaves form a complete subcategory $\mathrm{Shvs}(X, \V)$ in $\mathrm{PreShvs}(X, \V)$.
\end{remark}

\begin{definition}
    Diagram on a partially ordered set $S$ with values in category $V$ is functor $D: \cat(S) \to V$
\end{definition}

Category $\mathrm{Diag}(S, \V)$ of diagrams on partially ordered set $S$ is equivalent to category $\mathrm{Shvs}(X_S, \V)$ of sheaves on Alexandrov's topological space. \cite[Theorem 4.2.10.]{curry}

\subsection{Morphisms of sheaves.}

\begin{definition} \label{imF}
    Let $f: S \to T$ be a morphism of partially ordered sets $S, T$. Let $\G \in \mathrm{Shvs(X_T, \V}$ be a sheaf defined on $X_T$. Then \textit{pullback} of $\G$ is defined as following presheaf $f^\ast \G$ on $S$

    $$f^\ast \G(U) = G(f(U)), \quad f^{\ast}\G(U \supseteq V) = \G(f(U) \supseteq f(V))$$
\end{definition}

\begin{fact}
    $f^{\ast} \G$ defined in \ref{imF} is a sheaf on $X_S$ and mapping $f^{\ast}$ is functorial.
\end{fact}

If $f: S \hookrightarrow T$ is exact embedding it is natural to denote $f^\ast \G$ as $\G|_{S}$.

\subsection{Sheaf cohomology.}

\begin{definition}
    \textit{Functor of global sections} $\Gamma$ is a functor from $\mathrm{Shvs}(X_S, \V)$ to category $\V$ such that

    $$\Gamma: \F \to \F(X_S) = \lim_{\xleftarrow{s \in S}} \F(s)$$
\end{definition}

\begin{construction} \label{sWConstr}
    Let $s$ an element of finite poset $S$, $W$ be an from category $\V$. Denote by $\lceil s \rceil_W$ following sheaf over $S$

    $$\lceil s \rceil_W(t) = \begin{cases}
        W, & \text{if $t \leq s$}\\
        0, & \text{otherwise.}
    \end{cases}$$
\end{construction}

\begin{fact} \label{inj_shv}
    If $W$ is injective object in $\V$ then $\lceil s \rceil_W$ is injective in $\mathrm{Shvs}(S)$. \cite[Lemma 7.1.15]{curry}
\end{fact}

\begin{construction} \label{shvMon}
    Consider the following morphism of sheaves

    $$A: \F \to \bigoplus_{s \in S} \lceil s \rceil_{M_s} $$
    where $M_s$ is some injective module in $\V$ with monomorphism $i_s: \F(s) \to M_s$ and $A = \oplus_{s \in S} A_s$ where $A_s$ is formed naturally from $i_s$.
\end{construction}

\begin{fact}
    Morphism constructed in \ref{shvMon} is monomorphism, i.e. $\bigoplus_{s \in S} \lceil s \rceil_{M_s}$ is injective object in $\mathrm{Shvs}$. \cite[Lemma 7.1.15]{curry}
\end{fact}

\begin{definition} \label{godeman}
    \textit{Injective Godement's resolution} of sheaf $\F$ is defined as following exact sequence

    $$0 \xrightarrow{} \F \xrightarrow{i_0} \I^{0} \xrightarrow{i_1} \I^{1} \xrightarrow{i_2} \I^{2} \xrightarrow{i_3} \ldots $$

        where $\I^0 = \bigoplus_{s \in S} \lceil s \rceil_{M_s}$ and $i_0$ is taken from construction \ref{shvMon} and $\I^k, i_k$ follows by applying construction \ref{shvMon} to $\mathrm{coker} \ i_{k - 1}$.
\end{definition}

\begin{definition}
    \textit{Cohomology} $H^{\star}(S, \F)$ of sheaf $\F$ are defined as derived functors taken from functor of global sections $\F$. In other words, consider following differential complex

    $$0 \xrightarrow{} \I^0(X_S) \xrightarrow{i_1(S)} \I^1(X_S) \xrightarrow{i_2(S)} \I^2(X_S) \xrightarrow{i_3(S)} \ldots ,$$

    where $\I^j(X_S)$ are global sections. Then

    $$H^j(X_S, \F) = \ker i_{j + 1}(S) / \mathrm{im} \ i_j(S)$$
\end{definition}

There is also another approach to calculating cohomology of a poset embedded with a sheaf: constructing simplicial complex with local coefficient system and then calculating simplicial cohomology. In following sections we introduce necessary definitions and theorems describing link between regular and algorithmic approaches to calculating cohomology.

\subsection{(Co)chain complexes}

Recall that by $\V$ we denote an abelian category, for example $\Abel$ or $\ko\Vect$.

The sequence
\begin{equation}\label{eqCochainCpx}
0\to C^0\stackrel{d_0}{\to} C^1 \stackrel{d_1}{\to} C^2 \stackrel{d_2}{\to}\cdots
\end{equation}
of objects of $\V$ and morphisms between them is called a cochain complex if $d_j;d_{j+1}=0$ for any $j=0,1,\ldots$. For short, a chain complex~\eqref{eqCochainCpx} is denoted $(C^\bullet,d)$.

The definition implies the natural inclusion $p_j\colon \im d_j\hookrightarrow \ker d_{j+1}$.

\begin{definition}\label{definCohomologyCochain}
The cohomology of the cochain complex $(C^\bullet,d)$ are defined by
\[
H^j(C^\bullet,d)=\ker d_{j+1}/\im d_j = \Coker p_j \mbox{ for }j=0,1,\ldots.
\]
\end{definition}

A cochain complex $(C^\bullet,d)$ --- or any sequence of the form~\eqref{eqCochainCpx} --- is called exact (resp. exact at position $i$), if $H^j(C^\bullet,d)=0$ for any $j$ (resp. for $j=i$). Equivalently, $\im d_j=\ker d_{j+1}$.

\begin{lemma}[Zig-zag lemma]\label{lemZigZag}
A short exact sequence $0\to (C^\bullet,d_C)\to (B^\bullet,d_B)\to (A^\bullet,d_A)\to 0$ in $\Cochain(\V)$ induces the long exact sequence of cohomology
\begin{multline}\label{eqLongExactCochain}
    0\to H^0(C^\bullet,d_C)\to H^0(B^\bullet,d_B)\to H^0(A^\bullet,d_A)\to\\
    \to H^1(C^\bullet,d_C)\to H^1(B^\bullet,d_B)\to H^1(A^\bullet,d_A)\to\\
    \to H^2(C^\bullet,d_C)\to H^2(B^\bullet,d_B)\to H^2(A^\bullet,d_A)\to \cdots
\end{multline}
\end{lemma}

\subsection{Simplicial approach to cohomology}

Let $K$ be a finite simplicial complex. Local (cohomological) coefficient system $\F$ on $K$ is defined as functor from $\cat(K)$ to $\V$. That is, local coefficient system assigns to each non-empty simplex $I\in K$ an object of category $\V$, and to each inclusion $J\subset I$ a homomorphism $\F(J\subset I)\colon \F(J)\to \F(I)$ so that all possible diagrams commute. 

Define the cochain complex with coefficients in the local system $\F$ on $K$:

\begin{equation}\label{eqComplexLocalCoeffs}
0\to C^0(K;\F)\to C^1(K;\F)\to C^2(K;\F)\to\cdots
\end{equation}
where
\begin{itemize}
  \item $C^j(K;\F)=\bigoplus_{J\in K,\dim J=j}\F(J)$;
  \item Differential $d\colon C^j(K;\F)\to C^{j+1}(K;\F)$ is defined by the formula
  \[
  d=\bigoplus_{\substack{I\subset J\\\dim J=\dim I+1}}\inc{I}{J}\F(I\subset J).
  \]
\end{itemize}

\begin{definition}
    
The cohomology of the complex \eqref{eqComplexLocalCoeffs} is called the cohomology of the simplicial complex $K$ with cofficients in the system $\F$. We will denote it by $\H_{al}^*(K;R)$:
\[
\H_{al}^j(K;R)=\ker(d\colon C^j(K;\F)\to C^{j+1}(K;\F))/\Im(d\colon C^{j-1}(K;\F)\to C^{j}(K;\F)).
\]
\end{definition}

Above construction allows us to compute cohomology of a sheaved space in two different ways. Let us describe the connection between sheaf and algorithmic cohomology. The following theorem is taken from ~\cite[Thm.7.3.2]{curry}, where a dual version for homology and cosheaves is proven.

\begin{theorem}\label{thmSheafAlgorithmicSimpComp}
Let $K$ be a simplicial complex, $\F$ be a sheaf (or cohomological local coefficient system) on $K$. Then there is a natural isomorphism $H^*(K;\F)\cong H_{al}^*(K;\F)$.
\end{theorem}

\begin{construction} \label{rossConst}
Let $\sigma=(s_0<s_1<\cdots<s_k)$ be a chain in $S$, i.e., a simplex in $\ord(S)$. Denote by $\sigma_{\max}$ the element $s_k$. If $\sigma\subseteq\tau$ is a set-theoretic inclusion of chains, then clearly $\sigma_{\max}\leqslant \tau_{\max}$. For each sheaf $\F$ on $S$, we define a local coefficient system $\widehat{\F}$ on $\ord S$ by the formula
\[
\widehat{\F}(\sigma)=\F(\sigma_{\max}),\quad \widehat{\F}(\sigma\subseteq\tau)=\F(\sigma_{\max}\leqslant \tau_{\max}).
\]
\end{construction}

\begin{definition}
    The cochain complex $C_{roos}(S;F)$ is defined as $C_{al}(\ord(S); \hat{F})$. Its cohomology is denoted by $H_{roos}^*(S;F)$.
\end{definition}

\begin{remark}
    The complex $C_{roos}$ was first defined in the work of J. Roos \cite{roos}, hence the name.
\end{remark}

\begin{theorem}\label{thmGeneralSheafToLocCoeffSyst}
The sheaf cohomology $H^*(S;\F)$ are isomorphic to the cohomology of Roos complex $H_{roos}^*(S; F)$.
\end{theorem}

The proof can be found in \cite{curry}.

\section{Beat collapsing and cohomology invariance} \label{S3}

Let $(S, \F)$ be a sheaved space and $v \in S$ be a beat element. Denote by $S' = S \setminus \{ v \}$. Then $(S',{\F|_{S'}})$ is constructed by \textit{collapsing of beat} $v$. In this section we show that beat collapsing does not change cohomology of $(S, \F)$. 

\begin{remark}
    Let us briefly discuss functoriality of cohomology of sheaved spaces. Let $f: (S, \F) \to (T, \G)$ be a mapping of sheaved spaces. Recall that by definition $f = (a, \A)$ consists of two mappings.
$$
a: S \hookrightarrow T 
$$
$$
\A: a^\ast \G \to \F
$$

where $a^\ast \G$ denotes pullback of sheaf $\G$. Then we can decompose $f$ into two mappings.

$$
h: (S, \F) \to (S, a^\ast \G), \ \text{and} \ g: (S, a^\ast \G) \to (T, \G).
$$

$h$ is just morphism of sheaves $\F \to \G$ on the same topological space. It is a known fact that such mappings are functorial in category $\mathrm{Shvs}(S)$. Functoriality of mapping $g$ directly follows from B. Iversen~\cite[Page 100]{iversen}.
\end{remark}

\begin{theorem} \label{cohom_inv}
    Let $v$ be a beat element of $(S, \F)$. Denote by $S' = S \setminus \{ v \}$. The canonical inclusion $S' \to S$ induces the natural isomorphism

    $$H^{\ast}(S, \F) \cong H^{\ast}(S', \F|_{S'}) $$
\end{theorem}

Before we prove \ref{cohom_inv} we need several lemmas.

\begin{lemma} \label{global_down}
    Let $(S, \F)$ be a sheaved space, $v \in S$ be a downbeat. Denote by $S' = S \setminus \{ v \}$. The canonical inclusion $S' \to S$ induces the natural isomorphism

    $$\Gamma(S, \F) \cong \Gamma(S', \F|_{S'})$$
\end{lemma}

\begin{proof}
    Recall that for any partially ordered set $T$ equipped with a sheaf $\G$ by construction $\Gamma(T, \G) \subseteq \prod_{t \in T} \G(t)$ consists of all elements $x$ such that for any $t \leq t'$ we have $\mathrm{res}^{t}_{t'}(p_t(x)) = p_{t'}(x)$, where $p_s: \prod_{t \in T} \G(t) \to \G(s)$ is the canonical projection.

    Consider the map $\iota: \prod_{t \in S} \F(t) \to \prod_{t \in S'} \F|_{S'}(t)$ uniquely defined by 
    
    $$p_{s'}(\iota(x)) = p_{s'}(x)$$

    for all $s' \in S'$. Let us check that $\iota(\Gamma(S, \F)) \subseteq \Gamma(S', \F|_{S'})$.

    $$\mathrm{res}^{t}_{t'}(p_t(\iota(x))) = \mathrm{res}^{t}_{t'}(p_t(x)) = p_{t'}(x) = p_{t'}(\iota(x))$$
    
    Let $u$ be the origin of edge incoming to $v$ in Hasse diagram notation. Consider the map $f: \prod_{t \in S'} \F|_{S'}(t) \to \prod_{t \in S} \F|_{S}(t)$ uniquely defined by 
    
    $$p_{s}(f(x)) = \begin{cases}
        \mathrm{res}^{u}_{v}(p_u(x)) & s = v \\
        p_s(x) & s \neq v
    \end{cases}$$

    for all $s \in S$. Let us check that $f(\Gamma(S', \F|_{S'})) \subseteq \Gamma(S, \F)$. 

    \textit{Case 1}. $t \neq v$.

    $$\mathrm{res}^{t}_{t'}(p_t(f(x))) = \mathrm{res}^{t}_{t'}(p_t(x)) = p_{t'}(x) = p_{t'}(f(x))$$

    \textit{Case 2}. $t = v$.

    $$\mathrm{res}^{t}_{t'}(p_t(f(x))) = \mathrm{res}^{v}_{t'}(p_v(f(x)))  = \mathrm{res}^{v}_{t'}(\mathrm{res}^{u}_{v}(p_u(x))) = \mathrm{res}^{v}_{t'}(p_v(x)) =p_{t'}(x) = p_{t'}(f(x))$$

    Let us proof that $\iota \circ f = \id$.

    $$p_{s'}((\iota \circ f) (x)) = p_{s'}(\iota(f(x)) = p_{s'}(f(x)) = p_{s'}(x)$$.

    Let us proof that $f \circ \iota = \id$

    \textit{Case 1}. $s \neq v$.

    $$p_s((f \circ \iota) (x)) = p_{s}(f(\iota(x))) = p_{s}(i(x)) = p_s(x)$$

    \textit{Case 2}. $s = v$.

    $$p_s((f \circ \iota) (x)) = p_v(f(\iota(x))) = \mathrm{res}_v^u(p_u(\iota(x))) = p_v(\iota(x)) = p_v(x)$$
\end{proof}

\begin{lemma} \label{global_up}
    Let $(S, \F)$ be a sheaved space, $v \in S$ be a upbeat. Denote by $S' = S \setminus \{ v \}$. The canonical inclusion $S' \to S$ induces the natural isomorphism

    $$\Gamma(S, \F) \cong \Gamma(S', \F|_{S'})$$
\end{lemma}

\begin{proof}
     Recall that for any partially ordered set $T$ equipped with a sheaf $\G$ by construction $\Gamma(T, \G) \subseteq \prod_{t \in T} \G(t)$ consists of all elements $x$ such that for any $t \leq t'$ we have $\mathrm{res}^{t}_{t'}(p_t(x)) = p_{t'}(x)$, where $p_s: \prod_{t \in T} \G(t) \to \G(s)$ is the canonical projection.

    Consider the map $\iota: \prod_{t \in S} \F(t) \to \prod_{t \in S'} \F|_{S'}(t)$ uniquely defined by 
    
    $$p_{s'}(\iota(x)) = p_{s'}(x)$$

    for all $s' \in S'$. Let us check that $\iota(\Gamma(S, \F)) \subseteq \Gamma(S', \F|_{S'})$.

    $$\mathrm{res}^{t}_{t'}(p_t(\iota(x))) = \mathrm{res}^{t}_{t'}(p_t(x)) = p_{t'}(x) = p_{t'}(\iota(x))$$

    Let $u$ be the tail of edge outgoing from $v$ in Hasse diagram notation. Recall that $\mathrm{res}^{v}_{u}$ is isomorphism by definition, hence $(\mathrm{res}^{v}_{u})^{-1}$ is well-defined. Denote by $h = (\mathrm{res}^{v}_{u})^{-1}$. Consider the map $f: \prod_{t \in S'} \F|_{S'}(t) \to \prod_{t \in S} \F|_{S}(t)$ uniquely defined by 
    
    $$p_{s}(f(x)) = \begin{cases}
        h(p_u(x)) & s = v \\
        p_s(x) & s \neq v
    \end{cases}$$

    for all $s \in S$. Let us check that $f(\Gamma(S', \F|_{S'})) \subseteq \Gamma(S, \F)$. 

    \textit{Case 1}. $t \neq v$.

    $$\mathrm{res}^{t}_{t'}(p_t(f(x))) = \mathrm{res}^{t}_{t'}(p_t(x)) = p_{t'}(x) = p_{t'}(f(x))$$

    \textit{Case 2}. $t = v$.

    $$\mathrm{res}^{t}_{t'}(p_t(f(x))) = \mathrm{res}^{v}_{t'}(p_v(f(x)))  = \mathrm{res}^{v}_{t'}(h(p_u(x))) = \mathrm{res}^{v}_{t'}(p_v(x)) = p_{t'}(x) = p_{t'}(f(x))$$

    Let us proof that $\iota \circ f = \id$.

    $$p_{s'}((\iota \circ f) (x)) = p_{s'}(\iota(f(x)) = p_{s'}(f(x)) = p_{s'}(x)$$.

    Let us proof that $f \circ \iota = \id$

    \textit{Case 1}. $s \neq v$.

    $$p_s((f \circ \iota) (x)) = p_{s}(f(\iota(x))) = p_{s}(i(x)) = p_s(x)$$

    \textit{Case 2}. $s = v$.

    $$p_s((f \circ \iota) (x)) = p_v(f(\iota(x))) = h(p_u(\iota(x))) = p_v(\iota(x)) = p_v(x)$$
\end{proof}

\begin{corollary} \label{crl}
    Let $(S, \F)$ be a sheaved space, $v \in S$ be a beat element. Denote by $S' = S \setminus \{ v \}$. The canonical inclusion $S' \to S$ induces the natural isomorphism

    $$\Gamma(S, \F) \cong \Gamma(S', \F|_{S'})$$
\end{corollary}

\begin{proof}
    Corollary directly follows from Lemmas \ref{global_down}, \ref{global_up}.
\end{proof}

Next step in proof is showing that if $\I$ is injective resolution of a sheaf $\F$ on $S$ then $\I|_{S'}$ is injective resolution of the sheaf $\F|_{S'}$ on $S'$. Construction \ref{shvMon} allows us to proof injectivity of only sheaves of form $\lceil s \rceil_{W}$, where $W$ is an injective object in $\V$. We will discuss it more formally later.

\begin{lemma} \label{swinj}
    Let $(S, \lceil s \rceil_{W})$ be a sheaved space, where $W$ is an injective object in $\V$, $v \in S$ be a beat element, $s \in S$. Denote by $S' = S \setminus \{ v \}$. Then $(\lceil s \rceil_{W})|_{S'}$ is an injective object of $\mathrm{Shvs}(S')$.
\end{lemma}

\begin{proof}
    Recall that by definition 

    $$\lceil s \rceil_W(t) = \begin{cases}
        W, & \text{if $t \leq s$}\\
        0, & \text{otherwise.}
    \end{cases}$$

    If $s \neq v$, by definition we get

    $$(\lceil s \rceil_W)|_{S'}(t) = \begin{cases}
        W, & \text{if $t \leq s$}\\
        0, & \text{otherwise.}
    \end{cases}$$

    which is injective by Fact \ref{inj_shv}. Let us discuss case when $s = v$, i.e. $s$ is the deleted vertex.

    \textit{Case 1.} $v$ is a downbeat.

    Since $v$ is a downbeat there exists exactly one $u$ such that exists arrow $e$ from $u$ to $v$ in Hasse diagram notation. Consider following sheaf $\G$ on $S'$.

    $$\G(t) = \lceil u \rceil_W(t) = \begin{cases}
        W, & \text{if $t \leq u$}\\
        0, & \text{otherwise.}
    \end{cases}$$

    Note that $\G$ is injective by Fact \ref{inj_shv} and $\G$ equals $(\lceil v \rceil_W)|_{S'}$. Indeed the only deleted vertex $v$ is a downbeat, therefore by definition $t \leq u$ for any $t \in S$ such that $t < v$.
    
    \textit{Case 2.} $v$ is an upbeat.

    Consider following diagram.

    \begin{equation}
        \begin{tikzcd}[row sep=2.5em]
        & \lceil s \rceil_W \\
        \A \arrow[r, "g'", hook] \arrow[ur, "h'"] &
        \B \arrow[u, "\varphi"', dotted] \\
        \end{tikzcd}
        \label{diag:up}
    \end{equation}

    \vspace{-20pt}

    $\varphi$ exists since $\lceil s \rceil_W$ is injective object in $\mathrm{Shvs}(S)$. Our goal is to show the existence of $\psi$ in diagram below.

    \[
    \begin{tikzcd}[row sep=2.5em]
    & (\lceil s \rceil_W)|_{S'} \\
    \C \arrow[r, "g", hook] \arrow[ur, "h"] &
    \D \arrow[u, "\psi"', dotted] \\
    \end{tikzcd}
    \]
    
    \vspace{-20pt}

    It is enough to show that any sheaf $\C$ on $S'$ can be described as restriction of some sheaf $\A$ on $S$ onto $S'$. Since $v$ is a upbeat there exists exactly one $u$ such that exists arrow $e$ from $v$ to $u$ in Hasse diagram notation. Now let us construct $\A, \B$ as follows.

    $$
    \A(s) = \begin{cases}
        \C(s), & \text{if $t \neq v$}\\
        \C(u), & \text{otherwise.}
    \end{cases}
    \hspace{20pt}
    \B(s) = \begin{cases}
        \D(s), & \text{if $t \neq v$}\\
        \D(u), & \text{otherwise.}
    \end{cases}
    $$

    Then by construction $\A|_{S'} = \C$ and $\B|_{S'} = \D$. Consider mappings $g': \A \to \B, \ h': \A \to \lceil s \rceil_W$ defined as follows.

    $$
    g'(\A(t)) = \begin{cases}
        g(\C(t)), & \text{if $t \neq v$}\\
        g(\C(u)), & \text{otherwise.}
    \end{cases}
    \hspace{20pt}
    h'(\A(t)) = \begin{cases}
        h(\C(t)), & \text{if $t \neq v$}\\
        h(\C(u)), & \text{otherwise.}
    \end{cases}
    $$

    It is obvious that resulting diagram commutes. It means that there exists $\varphi: \B \to \lceil s \rceil_W$ by definition of injective object. That allows us to apply the restriction functor $S \hookrightarrow S'$ to diagram \ref{diag:up} and therefore construct $\psi$ as restriction of $\varphi$.

\end{proof}

\begin{lemma} \label{injres}
    Let $v \in S$ be a beat element. Denote by $S' = S \setminus \{ v \}$. If $\I$ is injective resolution of a sheaf $\F$ on $S$ then $\I|_{S'}$ is injective resolution of the sheaf $\F|_{S'}$ on $S'$.
\end{lemma}

\begin{proof}
     Lemma \ref{swinj} implies that $I|_S'$ consists of injective objects in category $\mathrm{Shvs}(S')$. We also know that $I|_S'$ is a differential complex, i.e. $d^2 = 0$ since $I$ is a differential complex.

    The last thing to check is whether $I|_S'$ is exact. Recall that exactness of sheaves sequence is equivalent to exactness of stalk sequence for each point. Space $S' \subset S$ and all stalk sequences are exact for $S$, therefore all stalk sequences are exact for all points in $S'$.
\end{proof}

Now we are ready to proof main result of this section.

\begin{proof}[Proof of Theorem \ref{cohom_inv}]
    Let following differential complex be the Godement's injective resolution of $\F$. 

    $$
    0 \to \F \to \I_0 \to \I_1 \to \I_2 \to \dots 
    $$

    Note that it is exact and each $\I_i$ is injective. Applying functor of global sections $\Gamma$ to it gives us following differential complex $C_S$.

    \begin{equation}
    0 \xrightarrow[]{d} \Gamma(X_S, \F) \xrightarrow[]{d} \Gamma(X_S, \I_0) \xrightarrow[]{d} \Gamma(X_S, \I_1) \xrightarrow[]{d} \dots 
    \label{injresseq}
    \end{equation}
    
    Recall that by definition

    $$
    H^{\ast}(X_S, \F) = H^{\ast}(C_S, d).
    $$

    Consider the sequence $\I|_{S'}$ defined as restriction of sequence \ref{injresseq}.

    $$
    0 \to \F|_{S'} \to \I_0|_{S'} \to \I_1|_{S'} \to \I_2|_{S'} \to \dots 
    $$

    Lemma \ref{injres} implies that $\I|_{S'}$ is an injective resolution of $\F|_{S'}$. This means that it can be used to compute the cohomology of $\F|_{S'}$. Let us apply the global sections functor $\Gamma$ to it, obtaining the following differential complex $C_{S'}$.

    $$
    0 \xrightarrow[]{d} \Gamma(X_{S'}, \F|_{S'}) \xrightarrow[]{d'} \Gamma(X_{S'}, \I_0|_{S'}) \xrightarrow[]{d'} \Gamma(X_{S'}, \I_1|_{S'}) \xrightarrow[]{d'} \dots 
    $$

    Corollary \ref{crl} establishes that $\Gamma(X_{S'}, \I_i|_{S'}) \cong \Gamma(X_S, \I_i)$. It means that $C_S = C_{S'}$ as spaces. The equality of differentials $d = d'$ follows from the naturality of the isomorphisms of global sections. Therefore

    $$
    H^{\ast}(X_S, \F) = H^{\ast}(C_S, d) \cong H^{\ast}(C_{S'}, d') = H^{\ast}(X_{S'}, \F|_{S'}).
    $$
\end{proof}

\section{Сore of sheaved space} \label{S4}

\begin{theorem} \label{isom}
    All cores of a sheaved space $(S, \F)$ are isomorphic.
\end{theorem}

\begin{proof}
    We will construct the proof using Diamond lemma (or Newman's lemma). Consider an abstract rewriting system $\T$, where elements are sheaved spaces and applying arrow is collapsing of the beat. Starting element is our sheaved space $(S, \F)$.

    Well-foundness of the relation is obvious since $S$ is finite, hence contains only finite amount of beat elements. Let us show that $\T$ is locally confluent, i.e. if we can collapse beat element $v$ and beat element $u$ on some sheaved space $(T, \G)$, there exists such $(T', {\G'})$ that it can be reached by sequence of beat collapses from $(T \setminus \{ v \}, \G|_v)$ and $(T \setminus \{u\}, \G|_u)$. 

    Indeed, if $u$ and $v$ are not connected, e.g. there is no edge $e$ such that tail of $e$ is $u (v)$ and origin of $e$ is $v (u)$, they cannot change beatness of each other. That means that following diagram exists (as beat collapses are present).

    \[ 
    \begin{tikzcd}[row sep=1.5em]
        (T, \G) \arrow[r] \arrow[d] &
        (T \setminus \{ v \}, \G|_{T \setminus \{ v \}}) \arrow[d] \\
        (T \setminus \{ u \}, \G|_{T \setminus \{ u \}}) \arrow[r]&
        (T \setminus \{ {u, v} \}, \G|_{T \setminus \{ u, v \}}) \\
    \end{tikzcd}
    \]

    Without loss of generality, assume that there exists an edge $e$ such that tail of $e$ is $u$ and origin of $e$ is $v$. If $v$ is a downbeat or $u$ is a upbeat their collapses will not affect the type of the other. The final case to consider is when $v$ is a upbeat and $u$ is downbeat. However, since $v$ is a upbeat, $\mathrm{res}^v_u$ is an isomorphism. Therefore, collapsing either $u$ or $v$ will yield the same poset (up to isomorphism) and the same sheaf (also up to isomorphism). 

    Therefore Diamond lemma states that $\T$ is confluent all together. That means that it has exactly one final element up to isomorphism. This concludes the proof.
    
\end{proof}

\begin{remark}
    To demonstrate the canonical nature of the isomorphisms from Theorem \ref{isom}, we need to establish the functoriality of confluent arrows in the abstract rewriting system. We consider this a \textit{terra incognita} and a field for future research.

    Once this aspect is proven, the only change in Theorem \ref{result_main} would be the replacement of the article \textit{a} with \textit{the} before the term "core", as we would then be able to refer to \textit{the} core of a sheaved space.
\end{remark}

\begin{theorem}
    Let $(T, \G)$ be a core of $(S, \F)$. The canonical inclusion $T \hookrightarrow S$ induces the natural isomorphism

    $$H^{\ast}(X_T, \G) \cong H^{\ast}(X_S, \F)$$
\end{theorem}

\begin{proof}
    Recall that, by definition, core of sheaved space $(S, \F)$ can be obtained by consecutive deletion of vertices (beat collapses). Consider some sequence of beat collapses.

    $$
    (S, \F) \rightarrow (S_1, \F|_{S_1}) \rightarrow (S_2, \F|_{S_2}) \rightarrow \dots \rightarrow (T, \G)
    $$

    Theorem \ref{cohom_inv} implies that each of beat collapse and corresponding canonical inclusion induces the natural isomorphism of cohomology. That gives us following sequences.

    $$
    T \xhookrightarrow[g_k]{} S_{k} \xhookrightarrow[g_{k - 1}]{}  \dots \xhookrightarrow[g_2]{} S_2 \xhookrightarrow[g_1]{} S_1 \xhookrightarrow[g_0]{} S
    $$

    And corresponding induced canonical isomorphisms of cohomology.

    $$
    H^{\ast}(X_T, \G) \cong H^{\ast}(X_{S_k}, \F|_{S_k}) \cong \dots \cong H^{\ast}(X_{S_1}, \F|_{S_1}) \cong H^{\ast}(X_S, \F)
    $$

    Since inclusion $g: T \hookrightarrow S$ can be represented as composition $g_0 \circ g_1 \circ \dots \circ g_{k}$ it also induces natural isomorphism of cohomology

    $$H^{\ast}(X_T, \G) \cong H^{\ast}(X_S, \F).$$
\end{proof}

\section{Generalization of beat collapses} \label{S5}

In this section we present a generalization of Stong's beat collapses. We prove them for sheaved spaces and then discuss applications in regular posets.

\begin{definition}
    Let $S$ be a poset, $s \in S$ be a vertex. Let $S_{<s}$ be a subposet with induced ordering from $S$, consisting of all elements of $S$ which are strictly lower then $s$. In other words, 

    $$S_{<s} = \{ v \ | \ v \in S, v < s \}.$$

    We will call $S_{<s}$ \textit{a downset} of vertex $s$.
\end{definition}

\begin{definition}
    Let $S$ be a poset, $s \in S$ be a vertex. Let $S_{>s}$ be a subposet with induced ordering from $S$, consisting of all elements of $S$ which are strictly greater then $s$. In other words, 

    $$S_{>s} = \{ v \ | \ v \in S, v > s \}.$$

    We will call $S_{>s}$ \textit{an upset} of vertex $s$.
\end{definition}

\begin{theorem} \label{downSetThrm}
    Let $(S, \F)$ be a sheaved space. Let be $s$ be such vertex that $H^{\star}(|S_{<s}|)$ are trivial (i.e., $|S_{<s}|$ has the homology of a point). Then

    $$H^{\ast}(X_S, \F) \cong H^{\ast}(X_{S'}, \F|_{S'})$$

    where $S' = S \setminus \{ s \}$.
\end{theorem}

Before we prove \ref{downSetThrm} we need several lemmas.

\begin{lemma} \label{downsetTrivial}
    Let $S$ be a poset, $T \subseteq S$ be a lower order ideal in $S$, i.e., for any $t \in T$ and $s < t$ it holds that $s \in T$. Let $\G_T$ be the following sheaf over $S$:

    $$\G_T(s) = \begin{cases}
        W, & \text{if $s \in T$} \\
        0, & \text{otherwise.}
    \end{cases} $$
where $W$ denotes some object in category. Then $$H^{\ast}(X_S, \G_T) \cong H^{\ast}(|T|, W).$$
\end{lemma}

\begin{proof}
    Note that by Theorem \ref{thmGeneralSheafToLocCoeffSyst}

    $$H^{\ast}(X_S, \G_T) \cong H_{roos}^{\ast}(S, \G_T)$$

    And by definition (Construction \ref{rossConst}) we know that 

    $$H_{roos}^{\ast}(S, \G_T) = H^{\ast}(C_{roos}^{\ast}(S, \G_T), d_{al})$$

    Theorem \ref{thmSheafAlgorithmicSimpComp} implies the following isomorphism.

    $$H^{\ast}(C_{roos}^{\ast}(S, \G_T), d_{roos}) \cong H^{\ast}(C^{\ast}_{simp}(\ord T, \G_T), d_{simp})$$

    Which is by definition is simplicial cohomology of $\ord T$ with coefficients in $W$ and hence isomorphic to singular cohomology of $|T|$ with coefficients in $W$.

    $$H^{\ast}(C^{\ast}_{simp}(\ord T, \G_T), d_{simp}) = H^{\ast}_{simp}(\ord T, W) \cong H^{\ast}_{sing}(|T|, W)$$

    This concludes the proof.
    
\end{proof}

\begin{lemma} \label{lmPosetConst}
    Let $S$ be a poset, $s \in S$ - an arbitrary vertex, $W$ - some object in abelian category $\V$. Consider following sheaves over $S$.

    $$\low{\bar{s}}{W}(t) = \begin{cases}
        W, & \text{if $t < s$} \\
        0, & \text{otherwise.}
    \end{cases} \quad \quad 
    \delta_W(t) = \begin{cases}
        W, & \text{if $t = s$} \\
        0, & \text{otherwise.}
    \end{cases}$$

    Then $H^j(X_S, \low{\bar{s}}{W}) \cong H^{j-1}(X_S, \delta_W)$
\end{lemma}

\begin{proof}
    Consider following short exact sequence of sheaves over $S$.

    $$0 \to \low{\bar{s}}{W} \to \low{s}{W} \to \delta_W \to 0$$

    where $\bar{s}_{W}$ is defined by Construction \ref{sWConstr}. Applying Construction \ref{rossConst} gives us the following short exact sequence of differential complexes.

    $$0 \to C^{\ast}(S; \low{\bar{s}}{W}) \to C^{\ast}(S;\low{s}{W}) \to C^{\ast}(S;\delta_W) \to 0$$

    Indeed, from Definition \ref{rossConst} each $C^j_{roos}(S;\F)$ is a direct sum of stalks $\F(s)$ hence $C^*_{roos}(S; \cdot)$ is an exact functor. Zig-Zag Lemma \ref{lemZigZag} implies the existence of the induced long exact sequence in cohomology.

    $$ \dots \to H^j(C^{\ast}(S; \low{\bar{s}}{W})) \to H^j(C^{\ast}(S;\low{s}{W})) \to H^j(C^{\ast}(S;\delta_W)) \to H^{j + 1}(C^{\ast}(S; \low{\bar{s}}{W})) \to \dots$$

    We know that $\low{s}{W}$ is injective (\ref{inj_shv}), therefore $H^j(C^{\ast}(S;\low{s}{W})) \cong H^j(X_S, \low{s}{W}) = 0$ for all $j > 0$. Hence, we have following isomorphism.

    $$H^j(C^{\ast}(S;\delta_W)) \cong H^{j + 1}(C^{\ast}(S; \low{\bar{s}}{W})) $$

    Which corresponds by definition to isomorphism

    $$H^j(X_S, \low{\bar{s}}{W}) \cong H^{j-1}(X_S, \delta_W)$$
\end{proof}

Recall that $s$ is called the \textit{greatest} element of poset $S$, if for any $t\in S, \ t\leq s$

\begin{corollary}
    Let $T$ be a poset with the greatest element $s \in T$. Consider the following sheaf over $T$.

    $$\delta_W(t) = \begin{cases}
        W, & \text{if $t = s$} \\
        0, & \text{otherwise.}
    \end{cases}$$

    Then $\widetilde{H}^j(X_T, \delta_W) \cong \widetilde{H}^{j-1}(|T'|, W)$, where $T' = T \setminus \{ s \}$.
\end{corollary}

\begin{proof}
    Proof follows directly from Lemma \ref{lmPosetConst}.
\end{proof}

\begin{lemma} \label{trivialConst}
    Let $T$ be a poset such that $H_{\ast}(|T|, \Z)$ are trivial (i.e., $|T|$ has the same homology as a point). Then for any object $W$ in abelian category $\V$ the cohomology of a constant sheaf $W$ over $T$ are also trivial.
\end{lemma}

\begin{proof}
    This follows directly from Universal coefficient theorem in cohomology.
\end{proof}

Now we are ready to prove Theorem \ref{downSetThrm}.

\begin{proof}{Proof of Theorem \ref{downSetThrm}}

Consider sheaf $\F' = \F|_{S'}$ and following sheaf on $S$.

$$\delta_s(t) = \begin{cases}
        \F(s), & \text{if $t = s$} \\
        0, & \text{otherwise.}
    \end{cases}$$

Let us construct following short exact sequence of Roos cochain complexes.

$$0 \to C^{\ast}_{roos}(S', \F') \xrightarrow[]{f} C^{\ast}_{roos}(S, \F) \xrightarrow[]{g} C^{\ast}_{roos}(S, \delta_s) \to 0$$

Indeed, it is exact on stalks: if $t\neq s$ it takes the form $0 \to F(t) \to F(t) \to 0 \to 0$, and if $t=s$, it takes the form $0 \to 0 \to F(s) \to F(s) \to 0$. By Zig-zag Lemma \ref{lemZigZag} it induces the following long exact sequence in cohomology.

    $$ \dots \to H^j(C^{\ast}(S, \delta_s)) \to H^{j + 1}(C^{\ast}(S', \F')) \to H^{j + 1}(C^{\ast}(S, \F)) \to H^{j + 1}(C^{\ast}(S, \delta_s)) \to \dots$$

Denote by $\low{\bar{\G}}{s}$ following sheaf over $S$.

$$\low{\bar{\G}}{s}(t) = \begin{cases}
        \F(s), & \text{if $t < s$} \\
        0, & \text{otherwise.}
    \end{cases}$$

Note that $H^j(C^{\ast}(S, \delta_s)) \cong \widetilde{H}^{j - 1}(S, \low{\bar{\G}}{s})$ by Lemma \ref{lmPosetConst}. We also know that $\widetilde{H}^{j - 1}(S, \low{\bar{\G}}{s}) \cong \widetilde{H}^{j - 1}(|S_{<s}|, \F(s))$ according to Lemma \ref{downsetTrivial}. Since $\widetilde{H}^{j - 1}(|S_{<s}|, \Z)$ is trivial by assumption, applying Lemma \ref{trivialConst} shows us triviality of $\widetilde{H}^{j - 1}(|S_{<s}|, \F(s))$. Therefore $H^{j + 1}(C^{\ast}(S', \F')) \cong H^{j + 1}(C^{\ast}(S, \F))$.

Let us bring all of the above transformations together.

$$ H^j(C^{\ast}(S, \delta_s)) \overset{\text{Lemma \ref{lmPosetConst}}}{\cong} \widetilde{H}^{j - 1}(S, \low{\bar{\G}}{s}) \overset{\text{Lemma \ref{downsetTrivial}}}{\cong} \widetilde{H}^{j - 1}(|S_{<s}|, \F(s)) \overset{\text{Lemma \ref{trivialConst}}}{\cong} 0$$

\end{proof}

\begin{remark}
    Note that Theorem \ref{downSetThrm} coincides and generalizes the Theorem \ref{cohom_inv} (in particular, downbeat case). Indeed, if $v$ is downbeat vertex, then, by definition, there exists exactly one $u$ such that there exists arrow from $u$ to $v$ in Hasse diagram notation. Therefore, $S_{<v}$ has the greatest vertex, hence $|S_{<v}|$ is a cone, hence contractible and $\widetilde{H}_{\ast}(|S_{<v}|, \Z) = 0$. 
\end{remark}

Theorem \ref{downSetThrm} allows us to formulate following corollary for posets with constant coefficients.

\begin{corollary} \label{stongGeneralization}
    Let $S$ be a poset, $s \in S$ be such vertex that $H_{\star}(|S_{<s}|)$ are trivial (i.e., $|S_{<s}|$ has the homology of a point). Then

    $$H^{\ast}(X_S, \Z) \cong H^{\ast}(X_{S'}, \Z)$$

    where $S' = S \setminus \{ s \}$.
\end{corollary}

\begin{proof}
    This corollary directly follows from \ref{downSetThrm} by taking $\F$ as constant sheaf $\Z$.
\end{proof}

One might ask whether Corollary \ref{stongGeneralization} can be applied to acyclic \textit{upsets} ($|S_{>s}|$). Indeed, this is true, but only for regular posets without a sheaf. Let us now formulate and prove a homotopic version of this theorem. 

\begin{theorem} \label{homotopicStong}
    Let $S$ be a poset, $s \in S$ such vertex that either $|S_{<s}| \cong \pt$ or $|S_{>s}| \cong \pt$. Then $|S| \cong |S'|$, where $S' = S \setminus \{ s \}$.
\end{theorem}

\begin{proof}
    The proof can be found in \cite[Lemma 3.1]{barmak}.
\end{proof}

One is tempted to formulate Theorem \ref{homotopicStong} within the category of sheaved spaces. However, this requires first formalizing the notion of homotopic equivalence for sheaved spaces. We consider this a promising direction for future research and subsequent papers.

Following Lemmas \ref{joinHomology}, \ref{cupHomology} are well-known results; however, for the sake of completeness, we provide their proof. They will be used in the proof of the homology version of Theorem \ref{homotopicStong}.

\begin{lemma} \label{joinHomology}
    Let $A, B$ be topological spaces. Assume that either $A$ or $B$ are acyclic. Then $A \mathbin{\ast} B$ is also acyclic.
\end{lemma}

\begin{proof}
    To calculate the homology of $A \mathbin{\ast} B$ we will use Kunneth Formula for join \cite{spanier}, which states the following.

    $$\widetilde{H}_{r + 1}(A \mathbin{\ast} B) \cong 
    \sum_{i + j = r} \widetilde{H}_i(A) \otimes \widetilde{H}_j(B) \oplus \sum_{i + j = r - 1} \Tor(\widetilde{H}_i(A), \widetilde{H}_j(B))$$

    Note that $\widetilde{H}_i(A) \otimes \widetilde{H}_j(B) = 0$  and $\Tor(\widetilde{H}_i(A), \widetilde{H}_j(B)) = 0$ since by assumption either $\widetilde{H}_{\ast}(A) = 0$ or $\widetilde{H}_{\ast}(B) = 0$. Therefore $\widetilde{H}_{r + 1}(A \mathbin{\ast} B) = 0$.

\end{proof}

\begin{lemma} \label{cupHomology}
    Let $A, B$ be topological spaces, $C = A \cap B$. Assume that $H_{\ast}(C, \Z) = 0$. Then 
    
    $$H_{\ast}(A \underset{C}{\cup} B, \Z) \cong H_{\ast}(A \vee B, \Z)$$ 
    
    In particular, if $B = \Cone(C)$, then 
    
    $$H_{\ast}(A \underset{C}{\cup} B, \Z) \cong H_{\ast}(A \vee B, \Z) \cong {H}_{\ast}(A)$$
\end{lemma}

\begin{proof}
    Denote $A \cup B$ by $X$. Consider following Mayer–Vietoris exact sequence.

    $$\dots \to H_{j + 1}(X) \to H_j(C) \to H_j(A) \oplus H_j(B) \to H_j(X) \to H_{j - 1}(C) \to \dots$$

    By assumption $H_{\ast}(C) = 0$. Therefore, $H_j(A) \oplus H_j(B) \cong H_j(A \cup B)$. Since $H^{\ast}(A \vee B) \cong H_j(A) \oplus H_j(B) \cong H_j(A \cup B)$, which concludes the proof.

    The corollary follows directly from the facts that $H_{\ast}(A \vee B) \cong {H}_{\ast}(A)$ if $B$ is contractible, which holds since $B$ is a $\Cone$.
\end{proof}

\begin{theorem}[Homology version of Theorem \ref{homotopicStong}] \label{homologyStong}
    Let $S$ be a poset, $s \in S$ such vertex that either $\widetilde{H}_{\ast}(|S_{<s}|, \Z) = 0$ or $\widetilde{H}_{\ast}(|S_{>s}|, \Z) = 0$, (i.e., either $|S_{<s}|$ or $|S_{>s}|$ has the homology of a point). Then 
    
    $$\widetilde{H}_{\ast}(|S|, \Z) \cong \widetilde{H}_{\ast}(|S'|, \Z)$$
    
    where $S' = S \setminus \{ s \}$.
    
\end{theorem}

\begin{proof}
    
    Let $L$ denote the $\mathrm{link}_{\ord(S)}(s)$. $L$ consists of chains in $S$ that remain chains when we include $s$ in them. Therefore $L$ can be represented in a following form

    $$L = \mathrm{link}_{\ord(S)}(s) = \Big\{ \sigma = \sigma_{<} \sqcup \sigma_{>} \ \Big| \ 
    \begin{matrix}
        \sigma_{<} = (j_1 < \dots < j_k), \ j_k < s \\
        \sigma_{>} = (t_1 < \dots < t_n), \ t_n > s
    \end{matrix}
    \ \Big\}$$

    Now let us use the definition of a join of simplicial complexes.

    $$\{ \sigma = \sigma_{<} \sqcup \sigma_{>} \ | \dots \} = \{ \sigma_{<} \} \mathbin{\ast} \{ \sigma_{>} \}$$

    Considering the geometric realization of $L$, we obtain the following.
    
    $$|L| = |\{ \sigma_{<} \} \mathbin{\ast} \{ \sigma_{>} \}| = |\ord S_{<s}| \mathbin{\ast} |\ord S_{>s}|$$

    By assumption, either $|\ord S_{<s}|$ or $ |\ord S_{>s}|$ is acyclic. By Lemma \ref{joinHomology}, it implies that $|L|$ is also acyclic. Let us return to the original problem.

    $$|S| = |\ord S| = |\ord S'| \underset{|L|}{\cup} \Cone |L| $$

    Lemma \ref{cupHomology} implies that $\widetilde{H}_{\ast}(|S|, \Z) \cong \widetilde{H}_{\ast}(|S'|, \Z)$ since $|L|$ is acyclic. This concludes the proof.
    
\end{proof}

\begin{remark}
    Note that these simplifications present generalizations of Stong's beat collapses. There exist posets that lack beat vertices in the sense of Stong's definition, yet can still be simplified using Theorem \ref{homologyStong}. Let us provide an example.
\end{remark}

\begin{example}
    Consider the classic example of a contractible yet non-collapsible simplicial complex: Bing’s House \cite{bing}. Let $S$ be a poset whose geometric realization is Bing’s House. Construct a new poset $T$ by adding two vertices $u$ and $v$ to $S$, with the property that $s \leq u$ and $s \leq v$ for all $s \in S$, and such that $u$ and $v$ are incomparable with each other. In this construction, neither $u$ nor $v$ is an upbeat or a downbeat vertex. However, Theorem~\ref{homologyStong} implies that both can be removed without changing the cohomology, since Bing’s House is acyclic. Moreover, this holds for any sheaf $\mathcal{F}$ taking values in $\V$ over $T$.
\end{example}

\newcommand{\etalchar}[1]{$^{#1}$}

\end{document}